\documentclass{amsart}
\usepackage{fullpage}
\usepackage{graphicx}
\usepackage{amsmath,amsfonts,amssymb,latexsym,epic,eepic}
\usepackage{dsfont}
\usepackage[mathcal]{euscript}
\usepackage{multicol}
\usepackage{tikz}
\usetikzlibrary{patterns}

\newtheorem{theorem}{Theorem}[section]
\newtheorem{lemma}[theorem]{Lemma}

\newtheorem{remark}{Remark}[section]

\newcommand{\xN}{\mathbb{N}}
\newcommand{\xR}{\mathbb{R}}
\newcommand{\xH}{{\rm H}}

\newcommand{\bfC}{{\boldsymbol C}}
\newcommand{\bfE}{{\boldsymbol E}}
\newcommand{\bfH}{{\boldsymbol H}}
\newcommand{\bff}{{\boldsymbol f}}
\newcommand{\bfu}{{\boldsymbol u}}
\newcommand{\bfuini}{{\bfu_0}}
\newcommand{\bfv}{{\boldsymbol v}}
\newcommand{\bfw}{{\boldsymbol w}}
\newcommand{\bfe}{{\boldsymbol e}}
\newcommand{\bfn}{\boldsymbol n}
\newcommand{\bfx}{\boldsymbol x}
\newcommand{\bfX}{\boldsymbol X}

\newcommand{\dgammax}{\ \mathrm{d}\gamma(\bfx)}
\newcommand{\dx}{\ \mathrm{d}\bfx}
\newcommand{\dt}{\ \mathrm{d} t}

\newcommand{\vi}{v_i}

\newcommand{\mesh}{{\mathcal M}}
\newcommand{\edge}{{\sigma}}
\newcommand{\edgeperp}{{\tau}}
\newcommand{\edges}{{\mathcal E}}
\newcommand{\edgesK}{\edges(K)}

\newcommand{\edgesint}{{\mathcal E}_{\mathrm{int}}}
\newcommand{\edgesext}{{\mathcal E}_{\mathrm{ext}}}
\newcommand{\edgesinti}{{\mathcal E}_{\mathrm{int}}^{(i)}}

\newcommand{\edgesexti}{{\mathcal E}_{\mathrm{ext}}^{(i)}}
\newcommand{\edgesi}{{\edges\ei}}
\newcommand{\edgesj}{{\edges\ej}}
\newcommand{\edged}{\epsilon}
\newcommand{\edgesd}{{\widetilde {\edges}}}
\newcommand{\edgesdi}{{\edgesd^{(i)}}}
\newcommand{\edgesdij}{{\edgesd^{(i,j)}}} % ceux d edgedi orthogonaux a ej

\newcommand{\edgesdinti}{{\edgesd^{(i)}_{{\rm int}}}}

\newcommand{\ei}{^{(i)}}
\newcommand{\ej}{^{(j)}}

\newcommand{\Fcvedge}{F_{K,\sigma}}
\newcommand{\ucvedge}{u_{K,\sigma}}

\newcommand{\Hmeshzero}{\bfH_{\!\edges,0}}
\newcommand{\Hmeshmzero}{\bfH_{\!\edges_m,0}}

\newcommand{\Hmeshi}{H_{\!\edges^{(i)}}}
\newcommand{\Hmeshizero}{H_{\!\edges^{(i)},0}}

\newcommand{\llbracket}{\bigl[ \hspace{-0.55ex} |}
\newcommand{\rrbracket}{| \hspace{-0.55ex} \bigr]}
\newcommand{\iinud}{i \in \llbracket 1, d \rrbracket}

\newcommand{\dive}{{\mathrm{div}} }
\newcommand{\nablai}{{\boldsymbol \nabla}}

\newcommand{\ie}{\emph{i.e.\/}}
\newcommand{\mnn}{{m\in\xN}}

\newcommand{\characteristic}{\mathds{1}}

\begin{document}
\title[Convergence of the MAC scheme for variable density flows]
{Convergence of the MAC scheme for variable density flows}

\author{T. Gallou\"et}
\address{I2M UMR 7373, Aix-Marseille Universit\'e, CNRS, \'Ecole Centrale de Marseille.}
\email{thierry.gallouet@univ-amu.fr}

\author{R. Herbin}
\address{I2M UMR 7373, Aix-Marseille Universit\'e, CNRS, \'Ecole Centrale de Marseille.}
\email{raphaele.herbin@univ-amu.fr}

\author{J.C. Latch\'e}
\address{Institut de Radioprotection et de S\^uret\'e Nucl\'eaire (IRSN), Saint-Paul-lez-Durance, 13115, France.}
\email{jean-claude.latche@irsn.fr}

\author{K. Mallem}
\address{University of Skikda, Algeria.}
\email{khadidjamallem@gmail.com}

\subjclass[2000]{ }

\begin{abstract}
We prove in this paper the convergence of an semi-implicit MAC scheme for the time-dependent variable density 
Navier-Stokes equations.
\end{abstract}
\maketitle
%
% --------------------------------------------------------------------------------------------------------------------
%
\section{Introduction}

Let $\Omega$ be a parallelepiped of $\xR^d$, with $d\in\lbrace2,3\rbrace$ and $T>0$, and consider the following variable density Navier-Stokes equations posed on $\Omega\times (0, T )$:
\begin{subequations} \label{pb:cont}
\begin{align}\label{mass} &
\partial_t \bar\rho+\dive(\bar\rho\,\bar\bfu)=0,
\\ \label{qdm} & 
\partial_t( \bar\rho \, \bar \bfu)+\dive(\bar\rho \ \bar\bfu \otimes \bar\bfu)-\Delta\bar\bfu+\nablai \bar p=\bff,
\\ \label{inc} &
\dive\ \bar\bfu=0, 
\end{align} 
\end{subequations}
where $\bar\rho$, $\bar\bfu$ and $\bar p$ are the density, the velocity and the pressure of the flow and $\bff \in L^2(0,T;L^2(\Omega)^d)$.
This system is complemented with initial and boundary conditions $ \bar\bfu|_{\partial {\Omega}}=0,$ $\bar\bfu|_{t=0}=\bfu_{0}$, $\bar\rho|_{t=0}=\rho_{0}$, which are such that $\rho_{0} \in  L^{\infty}(\Omega)$, $0 <  \rho_{ \min} < \rho_{0} \le \rho_{\max} $ and $ \bfu_{0}\in L^{2}(\Omega)^{d}$.  
A pair $(\bar \rho, \bar \bfu)$ is a weak solution of problem \eqref{pb:cont} if it satisfies the following properties:
\begin{list}{--}{\itemsep=0.5ex \topsep=0.5ex \leftmargin=1.cm \labelwidth=0.3cm \labelsep=0.5cm \itemindent=0.cm}
\item $\bar \rho\in \{ \rho\in L^{\infty}(\Omega\times(0,T)), \ \rho>0 \ a.e. \ \mbox{in} \ \Omega\times(0,T)\}$.
\item $\bar \bfu\in\{ \bfu \in L^{\infty}(0,T;L^{2}(\Omega)^{d})\cap L^{2}(0,T;H^{1}_{0}(\Omega)^{d}), 
\ \dive \ \bfu=0 \ a.e.\ \mbox{in} \ \Omega\times(0,T)\}$.
\item For all $\varphi$ in $C_{c}^{\infty}(\Omega\times[0,T))$,
\begin{equation}  \label{weak-mass}
-\int_{0}^{T}\int_{\Omega} \bar \rho \partial_t \varphi +\bar \rho\, \bar \bfu \cdot \nabla \varphi \dx\dt
=\int_{\Omega}\rho_{0}(\bfx)\varphi(\bfx,0)\dx.
\end{equation}
\item For all $\bfv$ in $\displaystyle \{\bfw\in C_{c}^{\infty}(\Omega\times[0,T))^{d}, \dive \ \bfw=0\}$,
\begin{equation} \label{weak-mom}
\int_{0}^{T}\int_{\Omega} \bigr[ -\bar \rho\, \bar \bfu \cdot \partial_t \bfv
- (\bar \rho\, \bar \bfu \otimes \bar \bfu ):\nabla\bfv  
+ \nabla \bar \bfu :\nabla\bfv \bigl]\dx\dt =
\int_{\Omega}\rho_{0} \bfu_{0} \cdot\bfv(\cdot,0)\dx+\int_{0}^{T}\int_{\Omega}\bff \cdot \bfv \dx\dt.
\end{equation}
\end{list}
The existence of such a weak solution was proven in \cite{sim-90-non}; convergence results exist for  the discontinuous Galerkin approximation \cite{liu-07-con} and for a finite volume/finite element scheme \cite{LS_densitevar}. 
Here we prove the convergence of the MAC scheme.
%
% --------------------------------------------------------------------------------------------------------------------
% 
\section{The numerical scheme}\label{sec:discop}

Let $\mesh$ be a MAC mesh (see e.g. \cite{ghlm-16-macns}  and Figure \ref{fig:mesh} for the notations).
The  discrete pressure and density unknowns are associated with the cells of the mesh $\mesh$, and are denoted by $\big\{\rho_K,\ K \in \mesh \big\}$ and $\big\{p_K,\ K \in \mesh \big\}$.
The discrete velocity unknowns approximate the normal velocity to the mesh faces, and are denoted $(u_\edge)_{\edge\in\edgesi}$, $\iinud$, where $\edges$ is the set of the faces of the mesh, and  $\edges\ei$ the subset of the faces orthogonal to the $i$-th vector of the canonical basis of $\xR^d$.
We define $\edgesext=\{\edge\in\edges, \edge\subset \partial \Omega \}$, $\edgesint=\edges \setminus \edgesext$, $\edgesint\ei=\edgesint \cap  \edges\ei$ and $\edgesext\ei=\edgesext \cap  \edges\ei$.

\begin{figure}[tb]
 \centering
 \begin{tikzpicture}
	\fill[blue!20!white] (0.5,1.5)--(4.5,1.5)--(4.5,4)--(0.5,4)--cycle;
	\draw[thin,green,pattern=north east lines,pattern color=green!40!white] (2.5,1.5)--(6,1.5)--(6,4)--(2.5,4)--cycle;
	\path node at (0.6,3.5) [anchor= west]{$D_\edge$};
	\draw[very thin] (0.5,0.5)--(8.,0.5);
	\draw[very thin] (0.5,2.5)--(8,2.5);
	\draw[very thin] (0.5,5.5)--(8,5.5);
	\draw[very thin] (0.5,0.)--(0.5,6.);
	\draw[very thin] (4.5,0.)--(4.5,6.);
	\draw[very thin] (7.5,0.)--(7.5,6.);
	\draw[very thick] (0.5,2.5)--(4.5,2.5);
	\draw[very thick] (0.5,5.5)--(4.5,5.5);
	\draw[very thick] (4.5,2.5)--(7.5,2.5);

	% dual edges:
	\draw[very thick, red] (0.5,4)--(4.5,4);
	\draw[very thick, green!70!black] (4.5,1.5)--(4.5,4);
	\draw[very thick, blue] (0.5,1.5)--(0.5,4);
	\draw[very thin] (0.5,1.5)--(4.5,1.5);

	\path node at (0.6,5.1) [anchor= west]{$K$};
	\path node at (0.6,0.9) [anchor= west]{$L$};
	\path node at (3.7,2.7) {$\edge=K|L$};
	\path node at (7,2.7) {$\edge''$};
	\path node at (2.5,2.5) {$\times$};
	\path node at (2.5,5.5) {$\times$};
	\path node at (6,2.5) {$\times$};
	
	\path node at (2.5,5.2) {$\bfx_{\edge'}$};
	\path node at (2.5,2.2) {$\bfx_\edge$};
	\path node at (6.1,2.2) {$\bfx_{\edge''}$};
	
	\path node at (0.5,2.9) [anchor= west]{\textcolor{blue}{$\epsilon_2$}};
	\path node at (4.5,2.9) [anchor= west]{\textcolor{green!70!black}{$\epsilon_3$}};
	\path node at (3.88,5.3) {$\edge'$};
	\path node at (2.5,4.2) {\textcolor{red}{$\epsilon_1=\edge|\edge'$}};
	\path node at (0.5,-0.02) [anchor= north]{$\partial\Omega$};
	\draw[very thin, green!70!black, <->] (2.5,6.5)--(6,6.5); 
	\path node at (4.25,6.52) [anchor= south]{\textcolor{green!70!black}{$d_{\epsilon_3}$}};
	\path node at (5.5,3.5){\textcolor{green!50!black}{$D_{\epsilon_3}$}};
	\draw[very thin, blue, <->] (0.5,6.5)--(2.5,6.5);
	\path node at (1.5,6.52) [anchor= south]{\textcolor{blue}{$d_{\epsilon_2}$}};
	\draw[very thin, red, <->] (0,2.5)--(0,5.5); 
	\path node at (-0.02,4) [anchor= east]{\textcolor{red}{$d_{\epsilon_1}$}};
\end{tikzpicture}
\caption{Notations for control volumes and dual cells.}
\label{fig:mesh}\end{figure}
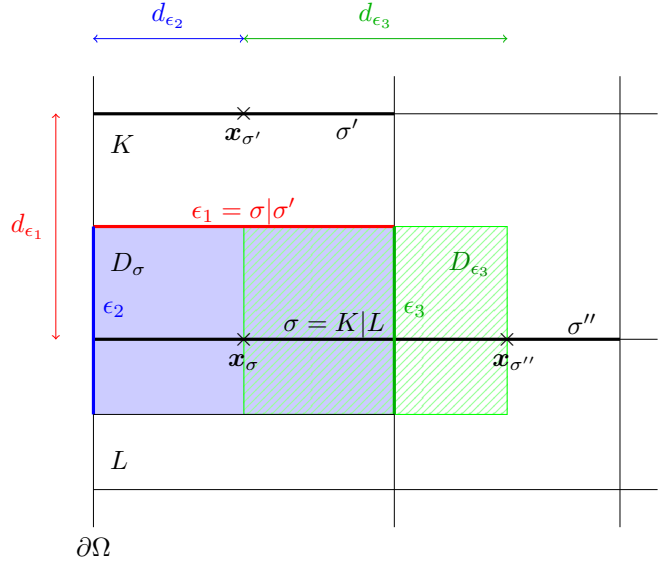

The regularity of the mesh is defined by:
\[
\eta_\mesh = \max\bigl\{ \frac{\vert \edge \vert}{\vert \edge' \vert},\ \edge \in \edgesi, \ \edge' \in \edgesj,\ i, j \in \llbracket 1, d\rrbracket, \ i\not = j \bigr\},
\]
and we denote by $h_\mesh$ the space step.
The discrete space $L_\mesh$ for the scalar unknowns (\ie\ the pressure and the density) is defined as the set of piecewise constant functions over each of the grid cells $K$ of $\mesh$, and the discrete space for the $i^{th}$ velocity component, $\Hmeshi$, as the set of piecewise constant functions over each of the grid cells $D_\edge,\ \edge\in\edgesi$.
The set of functions of $L_\mesh$ with zero mean value is denoted by $L_{\mesh,0}$.
As in the continuous case, the Dirichlet boundary conditions are (partly) incorporated into the definition of the velocity spaces:
\[
\Hmeshizero=\Bigl\{u\in\Hmeshi,\ u(\bfx)=0\ \forall \bfx\in D_\edge,\ \edge \in \edgesexti \Bigr\}, \quad \mbox{for } 1 \leq i \leq d
\]
(\ie\ we impose $u_\edge=0$ for all $\edge \in \edgesext$).
We then set $\Hmeshzero=\prod_{i=1}^d \Hmeshizero$.

\medskip
Let $0= t_{0} <t_{1}< \cdots < t_{N}=T$ be a partition of the time interval $(0,T)$, with $\delta t =t_{n+1}-t_n$.
Let $\{ u_{\edge}^{n+1}, \edge\in\edgesi, 0 \le n \le  N-1 ,\ 1\le i\le d\}$, $\{ p_{K}^{n+1}, K\in\mesh,\ 0 \le n\le  N-1\}$ and $\{\rho_{K}^{n+1}, K\in\mesh, 0 \le n \le  N-1 \}$ be the sets of discrete velocity, pressure and density unknowns.
Defining the characteristic function $\characteristic_A$ of any subset $A \subset \Omega$ by $\characteristic_A(\bfx)=1$ if $\bfx \in A$ and $\characteristic_A(\bfx)=0$ otherwise, the corresponding piecewise constant functions for the velocities are of the form:
\[
u_{i} = \sum_{n= 0}^{N-1} \sum_{\edge \in \edgesinti} u_{\edge}^{n+1}\characteristic_{D_\edge}\characteristic_{]t_n, t_{n+1}]},
\]
and  $X_{i,\edges,\delta t}$ denotes the set of such piecewise constant functions on time intervals and dual cells; we then set $\bfX_{\edges,\delta t} = \prod_{i=1}^d X_{i,\edges,\delta t}$. 
The pressure and density discrete functions are defined by:
\[
p = \sum_{n= 0}^{N-1} \sum_{K \in \mesh}p_{K}^{n+1} \characteristic_K\characteristic_{]t_n, t_{n+1}]},\qquad
\rho = \sum_{n= 0}^{N-1} \sum_{K \in \mesh} \rho_{K}^{n+1} \characteristic_K \characteristic_{]t_n, t_{n+1}]},  
\]  
and  $Y_{\mesh,\delta t}$ denotes the space of such piecewise constant functions.
The numerical scheme reads: 
\begin{subequations} \label{scheme}
\begin{align} \label{eq:scheme-init-rv} & 
\mbox{\bf Initialization: }\quad
\bfu^{(0)}= \widetilde{\mathcal{P}}_\edges \bfu_0, \quad \rho^{(0)} = \mathcal{P}_\mesh\rho_0.
\\ \nonumber & 
\mbox{\bf For } 0 \le n \le  N-1, \mbox{ solve for } \bfu^{n+1} \in \Hmeshzero,\ \rho^{n+1}\in L_\mesh \mbox{ and } p^{n+1}\in L_{\mesh,0}: 
\\[1ex] \label{scheme-mass} & \hspace{8ex}
\eth_t \rho^{n+1}+\dive_\mesh(\rho^{n+1}\bfu^n)=0,
\\[1ex] \label{scheme-mom} & \hspace{8ex}
\eth_t(\rho \bfu)^{n+1}+  {\bfC}_{\edges} (\rho^{n+1} \bfu^n)\ \bfu^{n+1}
- \Delta_{\edges} \bfu^{n+1} + \nabla_{\edges} \ p^{n+1} =\bff_{\edges}^{n+1}, 
\\[1ex] \label{scheme-div} & \hspace{8ex}
\dive_\mesh\bfu^{n+1} = 0,  
\end{align} 
\end{subequations}
with the interpolators and discrete operators defined as follows.

\medskip
\noindent {\bf Grid interpolators} -- The Fortin interpolator is defined by  $\widetilde{\mathcal{P}}_\edges \bfu = (\widetilde{\mathcal P}_{\edgesi})_{i=1,\ldots,d} $ with $\widetilde{\mathcal P}_\edgesi: H^1_0(\Omega) \longrightarrow \Hmeshizero$ and
\[
\vi \longmapsto \widetilde{\mathcal P}_\edgesi \vi = \sum_{\edge \in \edgesi} v_\edge\ \characteristic_{D_\edge}
\mbox{ with } v_\edge = \frac 1 {|\edge|} \int_\edge \vi \dgammax,\ \edge\in\edgesi.
\]
For $q\in L^2(\Omega)$,  $\mathcal{P}_\mesh q \in L_\mesh$ is defined by $\displaystyle \mathcal{P}_\mesh q (\bfx) = \frac 1 {|K|} \int_K q \dx$ for $\bfx \in K$.

\medskip
\noindent {\bf Discrete time derivative} --  For  $\rho \in  Y_{\mesh, \delta t}$, $\eth_t \rho  \in Y_{\mesh, \delta t}$ is defined by:
\[
\eth_t \rho (\bfx, t) = \sum_{n=0}^{N-1} \eth_t \rho^{n+1}(\bfx)\ \characteristic_{]t_n, t_{n+1}]}(t)
\mbox{ with}\quad
\eth_t \rho^{n+1} = \sum_{K\in \mesh}\frac 1 {\delta t}\ (\rho^{n+1}_K - \rho_K^n)\ \characteristic_K.
\]
 
\medskip
\noindent {\bf Discrete divergence} --
Let $\ucvedge$ be defined as $\ucvedge = u_\edge\, \bfn_{K,\edge} \cdot \bfe_i$ for any face $\edge \in \edgesi$, $i=1,\ldots, d$.
The discrete (upwind finite volume) divergence operator $\dive_\mesh$ is defined by:
\[ 
\dive_\mesh: \quad  L_\mesh \times \Hmeshzero \longrightarrow L_\mesh,
\quad (\rho,\bfu) \mapsto \dive_\mesh (\rho\bfu) = \sum_{K\in\mesh} \frac 1 {|K|} \sum_{\edge\in\edges(K)}  \Fcvedge\ \characteristic_K,
\]
with $\Fcvedge = | \edge |\ \rho_\edge \ucvedge$ for $K \in \mesh,\ \edge=K|L \in \edgesK$, and $\rho_\edge=\rho_K$  if $u_{K,\edge} \geq 0$, $\rho_\edge=\rho_L$ otherwise.
For all $K\in\mesh$, we set $(\dive \bfu)_K = \dive (1\times \bfu)_K$.

\medskip
\noindent {\bf Pressure gradient operator} --
The discrete pressure gradient operator is defined as the transpose of the divergence operator, so $\nabla_\edges: \ L_\mesh \longrightarrow \Hmeshzero,\ p \mapsto \nabla_\edges(p)$ with:
\begin{equation}\label{eq:def_grad_p}
\nabla_\edges p =\sum_{\edge=K|L \in \edgesint\ei} (\eth p)_\edge\, \bfn_{K,\edge}\ \characteristic_{D_\edge},
\mbox{ with} \quad (\eth p)_\edge=\frac{|\edge|}{|D_\edge|}  (p_L-p_K).
\end{equation}

\medskip
\noindent {\bf Discrete Laplace operator} --
The discrete diffusion operator $\Delta_\edges $ is defined in \cite{ghlm-16-macns} and is coercive in the sense that $ -\int_\Omega \Delta_{\edges} \bfv  \cdot \bfv \dx = \Vert\bfv\Vert_{1,\edges,0}^2$ for any $\bfv \in \Hmeshmzero$, where $\Vert\cdot\Vert_{1,\edges,0}$ is the usual discrete $\xH^1$-norm of $\bfu$ (see \cite{ghlm-16-macns}).
This inner product may also be formulated as the $L^2$-inner product of adequately chosen discrete gradients  \cite{ghlm-16-macns}.

\medskip
\noindent {\bf Discrete convection operator} --
The numerical convection fluxes and the approximations of $\rho$ in the momentum equation are chosen so as ensure that a discrete mass balance holds on the dual cells, in order to recover a discrete kinetic energy inequality.
This idea was first introduced in \cite{gal-08-unc,ans-11-anl} for the Crouzeix-Raviart and Rannacher-Turek scheme, in \cite{her-10-kin} for the MAC scheme and was adapted to a DDFV scheme \cite{gou-14-rhov}.
For $\edged = \edge | \edge'$, the convection flux $\int_{\edged} \rho u_{i} \bfu \cdot \bfn_{\edge,\edged}\dgammax$ is approximated by $F_{\edge,\edged} u_\edged,$ where $u_\edged=(u_\edge+u_{\edge'})/2$ and $F_{\edge,\edged}$ is the numerical mass flux through $\edged$ outward $D_{\edge}$ defined as follows:
\begin{list}{-}{\itemsep=0.ex \topsep=0.5ex \leftmargin=1.cm \labelwidth=0.7cm \labelsep=0.3cm \itemindent=0.cm}
\item First case -- The vector $\bfe_{i}$ is normal to $\edged$, and $\edged$ is included in a primal cell $K$.
Then the mass flux through $\edged=\edge | \edge'$ is given by:
\[
F_{\edge,\edged} = \frac 1 2  \bigl( F_{K,\edge}\ \bfn_{D_\edge,\edged} \cdot \bfn_{K,\edge}
+ F_{K,\edge'}\ \bfn_{D_\edge,\edged} \cdot \bfn_{K,\edge'} \bigr).
\]
\item  Second case -- The vector $\bfe_i$ is tangent to $\edged$, and $\edged$ is the union of the halves of two primal faces $\edgeperp$ and $\edgeperp'$ such that $\edge=K|L$ with $\edgeperp\in \edges(K)$ and $\edgeperp' \in \edges(L)$.
Then:
\[
F_{\edge,\edged}= \frac 1 2 \ (F_{K,\edgeperp}+ F_{L,\edgeperp'} ).
\]
\end{list}
\begin{remark}\label{rem:dual_flux}
In both cases,  for $\edged = \edge |\edge'$,  the mass flux $F_{\edge,\edged}$ may be written as $ F_{\edge, \edged}=|\edged| \rho_\edged \tilde u_\edged$, with $\rho_\edged=(\rho_\edge + \rho_{\edge'})$ and $\tilde u_\edged=(\rho_\edge u_\edge+\rho_{\edge'} u_{\edge'})/(\rho_\edge +\rho_{\edge'})$ in the first case, and $\rho_\edged = (|\edgeperp| \rho_\edgeperp +|\edgeperp'|\rho_{\edgeperp'}) /(|\edgeperp| + |\edgeperp'|)$ and $\tilde u_\edged = (|\edgeperp| \rho_\edgeperp u_\edgeperp +|\edgeperp '| \rho_{\edgeperp'} u_{\edgeperp'})/(|\edgeperp| \rho_\edgeperp +|\edgeperp'| \rho_{\edgeperp'})$ in the second case.
\end{remark}
With this expression of the flux, we may define a discrete divergence operator on the dual cells:
\[
\begin{array}{l|l}
\dive_\edgesi(\rho,\bfv): \quad
&
L_\mesh \times \Hmeshzero  \longrightarrow  L_\edges
\\ & \displaystyle \quad
(\rho, \bfv) \longmapsto  \dive_\edgesi(\rho,\bfv)= \sum_{\edge \in \edgesinti}  \dive_{D_\edge} (\rho\bfv) \ \characteristic_{D_\edge},
\\ & \quad \displaystyle
\mbox{with } \dive_{D_\edge} (\rho, \bfv) =\frac{1}{|D_\edge|} \sum_{\edged\in \edgesd(D_\edge)} F_{\edge, \edged},
\; \forall \edge \in \edgesinti.
\end{array}
\]

\medskip
For the definition of the time-derivative $\eth_t(\rho \bfu)$, an approximation of the density on the dual cell $\rho_{D_\edge}$ is defined as:
\[
|D_\edge|\, \rho_{D_\edge} = |D_{K,\edge}|\, \rho_K + |D_{L,\edge}|\, \rho_L, \qquad \edge\in\edgesint,\, \edge=K|L.
\]
With the above definitions, if $(\rho,\bfu) \in L_\mesh \times \bfX_{\edges,\delta t}$ satisfies the mass balance equation \eqref{scheme-mass}, then the following mass balance on the dual cells holds:
\begin{equation} \label{mass-dual}
\frac 1 {\delta t} (\rho_{D_{\edge}}^{n+1} -\rho_{D_{\edge}}^n) +\dive_{D_{\edge}} (\rho^{n+1}  \bfu^n)=0.
\end{equation}
Note that a discrete duality property also holds, in the sense that, for $1 \leq i \leq d$,
 \begin{equation} \label{ddfv-i}
\forall \rho \in L_\mesh, \forall \bfv \in \Hmeshzero, \forall w \in \Hmeshizero,
\int_\Omega  \dive_\edgesi (\rho, \bfv) w \dx = \int_\Omega   \rho    \bfv  \cdot  \nabla_{\!\!\!\edgesi} w \dx,
\end{equation} 
where $(\rho \bfv)_\edgesi$ and $\nabla_{\!\!\!\edgesi} w$  are vector valued functions of components:
\[
[(\rho \bfv)_\edgesi]_j = \sum_{\edged \in \edgesdij} \rho_\edged \tilde v_\edged \characteristic_{D_\edged},  \qquad
[(\nabla w)_\edgesi]_j  = \sum_{\edged \in \edgesdij, \edge = \overrightarrow{\edge |\edge'}} \frac{u_{\edge'} - u_\edge}{d_\edged} \characteristic_{D_\edged},  
\]
with $\rho_\edged$ and $\tilde v_\edged$ defined in Remark \ref{rem:dual_flux} and $\edgesdij =\{ \edged \in \edgesdi ; \edged \perp \bfe^{(j)}\}$.
We finally define the $i$-th component ${C}^{(i)}_\edges(\rho\bfu)$ of the non linear convection operator by:
\[
\begin{array}{l|l}
C\ei_\edges(\rho,\bfu):
\quad & \quad
\Hmeshizero  \longrightarrow  \Hmeshizero
\\ & \displaystyle \quad
v \longmapsto {C}^{(i)}_\edges(\rho,\bfu) v = \sum_{\edge \in \edgesdinti}  
\frac 1{|D_\edge|} \sum_{\substack{\edged \in \tilde{\edges} (D_\edge)\\  \edged=\edge|\edge'}} \!
F_{\edge,\edged}\  \frac{v_{\edge}+v_{\edge'}}2 \ \characteristic_{D_\edge}.
\end{array}
\]
and the full (\ie\ for all the velocity components) discrete convection operator $ {\bfC}_\edges(\rho,\bfu), \ \Hmeshzero \longrightarrow \Hmeshzero$  by 
$
{\bfC}_\edges (\rho,\bfu) \bfv = (C^{(1)}_\edges(\rho,\bfu) v_1, \ldots, C^{(d)}_\edges(\rho\bfu) v_d)^t.
$
Let $\bfE_\edges$ be the subspace of $\Hmeshzero$ of divergence-free functions (with respect to the discrete divergence operator).
By  H\"older's inequality and \cite[Lemma 3.9]{ghlm-16-macns}, there exists $C_{\eta_\mesh} >0$ (depending only on $\eta_\mesh$)  such that, $\forall\ (\rho,\bfu, \bfv,\bfw) \in L_\mesh \times \bfE_\edges \times\Hmeshzero^2$,
\[\begin{array}{ll} &
|\bfC_\edges(\rho\bfu)\bfv \cdot  \bfw| \leq
C_{\eta_\mesh}\, \|\rho\|_{L^{\infty}(\Omega)}\, \|\bfu\|_{L^4(\Omega)^d}\, \|\bfv\|_{L^4(\Omega)^d} \, \|\bfw\|_{1,\edges,0}
\\[1ex]
\mbox{and}
\quad &
|\bfC_\edges(\rho\bfu)\bfv \cdot  \bfw| \leq
C_{\eta_\mesh}\, \|\rho\|_{L^{\infty}(\Omega)}\, \|\bfu\|_{1,\edges,0}\, \|\bfv\|_{1,\edges,0} \, \|\bfw\|_{1,\edges,0}.
\end{array}
\]
%
% --------------------------------------------------------------------------------------------------------------------
% 
\section{Estimates and convergence analysis}

Since the velocity is divergence-free, the mass equation is a transport equation on $\rho$, so that, thanks to the upwind choice, the following estimate holds:
\begin{equation} \label{bound-density}
\rho_{\min}\leq \rho^{n+1}\leq \rho_{\max},
\end{equation}
and the $L^2$-norm of $\rho^{n+1}$ is lower than the $L^2$-norm of the initial data $\rho_0$, for $0 \leq n \leq N-1$.
In addition, thanks to \eqref{mass-dual}, any solution to the scheme \eqref{scheme} satisfies the following discrete kinetic energy balance, for $1 \le i \le d$,  $\edge\in\edgesi$, $0\leq n \leq N-1,$ 
\begin{multline} \label{eq:kinetic}
\frac 1 {2 \delta t} \bigl[\rho_{D_{\edge}}^{n+1}(u_{\edge}^{n+1})^2-\rho_{D_{\edge}}^n(u_{\edge}^n)^2\bigr] 
+\frac{1}{2 \vert D_{\edge}\vert} \sum_{\substack{\edged\in\tilde{\edges}(D_{\edge})\\ \edged=\edge|\edge'}}
     F_{\edge,\edged}(\rho^{n+1},u^n) u_{\edge}^{n+1} u_{\edge'}^{n+1}
\\
-(\Delta u)_{\edge}^{n+1} u_{\edge}^{n+1}
+(\eth p)_\edge^{n+1} u_{\edge}^{n+1}
-  f_{\edge}^{n+1} u_{\edge}^{n+1}
= - \dfrac{1}{2 \delta t} \rho_{D_{\edge}}^n \bigl(u_{\edge}^{n+1}-u_{\edge}^n \bigr)^2.
\end{multline}
From this inequality, we obtain estimates on the velocity.
For $\bfu\in \bfX_{\edges,\delta t}$ satisfying \eqref{scheme}, there exists $C > 0$ depending on $\bfuini$, $\rho_0$ and $\bff$ such that,
\begin{equation}
\|\bfu\|_{L^{2}(\Hmeshzero)} \! = \!\!\sum_{n=0}^{N-1}\!\! \! \delta t \! \Vert \bfu^{n+1} \Vert_{1,\edges,0}^2 \leq C  \mbox{ and }
\|\bfu\|_{L^{\infty}(L^{2})}\! = \!\!\!\!\max_{0\leq n \leq N-1} \!\!\Vert \bfu^{n+1} \Vert_{L^{2}(\Omega)^{d}} \!\leq \!C.  \label{estiLdeux:ro}
	\end{equation}
These estimates yields the existence of a unique solution to the scheme: indeed, the first equation may be solved separately for $\rho^{n+1}$ and is linear with repect to this unknown and, once $\rho^{n+1}$ is known, the last two equations are a linear generalized Oseen problem for $\bfu^{n+1}$ and $p^{n+1}$, which is uniquely solvable thanks to the {\em inf-sup} stability of the MAC discretization.  
The convergence of the scheme requires some time compactness. 
Contrary to the constant density case \cite{ghlm-16-macns}, there is no uniform estimate on the time derivative, and compactness is obtained thanks to the following lemma together with the Fr\'echet-Kolmogorov theorem.

\begin{lemma}[Estimate on the time translates of the velocity] \label{lem:translates}
 Let $\bfu \in X_{\edges,\delta t}$ and $\rho \in Y_{\mesh,\delta t}$ and let $\tau >0$ then 
\begin{equation} \label{translates}
\int_0^{T-\tau}  \!\!\!\!\int_\Omega  \!\vert \bfu(\bfx,t+ \tau) \!- \!\bfu(\bfx,t) \vert^2 \dx \dt \le C_{\eta_\mesh,T} \frac{\rho_{\max}}{\rho_{\min}}(\Vert \bfu \Vert^3_{L^{2}(\Hmeshzero)} \!\!+1) \sqrt {\tau + \delta t }
\end{equation}
where $C_{\eta_\mesh,T}>0 $ only depends on $\Omega$, $T$, $\bff$ and on the regularity of the mesh $\eta_\mesh$.
\end{lemma}

\begin{proof}
In the continuous case, see e.g. \cite[pages 444-452]{boy-06-ele}, the estimate \eqref{translates} is obtained by bounding the term $ \int_0^{T-\tau} \int_\Omega \left(\rho(\bfx,t) \bfu(\bfx,t+ \tau) -\rho(\bfx,t)  \bfu(\bfx,t) \right) \cdot\bfw(\bfx,t) \dt $
with $\bfw(\bfx,t) =  \bfu (\bfx,  t + \tau) -\bfu(\bfx,t)$.
However, in the context of the MAC scheme, the components of $\bfu$ are piecewise constant on different meshes so we need to treat the space indices separately.
For a given $i = 1, \ldots, d$, we denote by $u$ and $w$ the $i$-th  component of $\bfu$ and $\bfw$, and by $\widetilde \rho$ the  piecewise constant function defined by $\widetilde \rho (\bfx,t) = \rho_{D_\edge}^{n+1}$ for $ (\bfx,t) \in D_\edge \times [t_n t_{n+1})$.
We then wish to bound the terms
\begin{align*}
  & A\ei = \int_0^{T-\tau} (A_1^{(i)}(t) + A_2^{(i)}(t))\dt, \mbox{ with } \\
  & A\ei_1(t) = \int_\Omega \left( \widetilde \rho(\bfx,t+\tau) u(\bfx,t+\tau) -\widetilde\rho(\bfx,t)  u(\bfx,t)\right)  w(\bfx,t) \dx, \\
  & A\ei_2(t) = \int_\Omega \left(\widetilde \rho(\bfx,t) - \widetilde \rho(\bfx,t+\tau) \right) u(\bfx,t+ \tau)  w(\bfx,t) \dx.
\end{align*} 
For lack of space, we only deal here with the term $A_2^{(i)}(t)$. 
Thanks to the  mass balance on the dual cells \eqref{mass-dual} and to the discrete duality formula \eqref{ddfv-i} we have: 
\begin{align*}
A_2\ei(t)
&
 = \sum_{n=1}^{N-1} \delta t  \ \characteristic_{(t,t+\tau)}(t_n)
\int_\Omega \dive_\edgesi ( \tilde\rho^{n+1}  \bfu^n ) u(\cdot,t+\tau)  w(\cdot,t) \dx
\\ &
= \sum_{n=1}^{N-1} \delta t  \ \characteristic_{(t,t+\tau)}(t_n)
\int_\Omega   (\rho^{n+1}  u^n)_\edgesi  \nabla_{\!\!\!\edgesi}  (u(\cdot,t+\tau)  w(\cdot,t)) \dx.
\end{align*}
 Using H\"older's inequalities and the fact that $\displaystyle \sum_{n=1}^{N-1} \delta t\ \characteristic_{(t,t+\tau)}(t_n) \le \tau + \delta t$, 
\begin{align*}
A_2\ei(t)
&
\le  \rho_{\max}  \delta t \Bigl(\sum_{n=1}^{N-1} \Vert  u^n \Vert_{L^6} \Bigr)^{\frac 1 2}
\ \Bigl(\sum_{n=1}^{N-1} \characteristic_{(t,t+\tau)}(t_n) \Bigr)^{\frac 1 2}
\ \Vert \nabla_{\!\!\!\edgesi} (u(t+\tau)  w(t)) \Vert_{L^{\frac 6 5}}  
\\ &
\le | \Omega |^{\frac 1 6} \displaystyle \rho_{\max} \Vert u \Vert^{\frac 1 2}_{L^2(L^6)}
(\delta t+\tau)^{\frac 1 2} \Vert \nabla_{\!\!\!\edgesi} (u(t+\tau) w(t)) \Vert_{L^{\frac 3 2}}. 
\end{align*} 
Now, by H\"older's inequality,
\begin{multline*}
\Vert \nabla_{\!\!\!\edgesi} (u(\cdot,t+\tau)  w(\cdot,t))  \Vert_{L^{\frac 3 2}}
\le \Vert  (\nabla_{\!\!\!\edgesi}   u(\cdot,t+\tau) )  w(\cdot,t) \Vert_{L^{\frac 3 2}} 
+ \Vert u(\cdot,t+\tau)   \nabla_{\!\!\!\edgesi}  (w(\cdot,t))  \Vert_{L^{\frac 3 2}}
\hspace{10ex}\\
\le \Vert  \nabla_{\!\!\!\edgesi} u(\cdot,t+\tau) \Vert_{L^{2}}^2 + \Vert  w (\cdot,t)\Vert_{L^{6}}^2 +  \Vert  \nabla_{\!\!\!\edgesi}  w(\cdot,t)\Vert_{L^{2}}^2 + \Vert  u(\cdot,t+\tau)  \Vert_{L^{6}}^2. 
\end{multline*}
Therefore,  integrating  over $(0,T-\tau)$ yields that
\[
\int_0^{T-\tau}  A_2\ei(t) \dt \le
|\Omega|^{\frac 1 6}  \rho_{\max} [\tau + \delta]^{\frac 1 2}\Vert u \Vert_{L^2(L^6)} [\Vert u \Vert_{ L^2(L^6)}
+ \Vert w \Vert_{L^2(\Hmeshizero)} + \Vert w \Vert_{L^2(L^6)} + \Vert u \Vert_{L^2(\Hmeshizero)}].  
\]
Similar computations for the term $  \int_0^{T-\tau} A_1\ei (t) \dt$ yield the result.
\end{proof}

\begin{theorem}[Convergence of the scheme]
\label{conv:ins}
Let $(\delta t_{m})_{m\in \xN}$  and $ (\mesh_m)_\mnn$ be a sequence of 
time steps  and MAC grids such that $\delta t_{m} \rightarrow 0$
and $h_{\mesh_m} \to 0$  as $m \to +\infty$~; assume that there exists $\eta >0$ such that $\eta_{\mesh_m} \le \eta$ for any $m\in \mathbb N$.
Let $(\rho_m,\bfu_m)$ be a solution to  \eqref{scheme} for $\delta t=\delta t_{m}$ and $\mesh=\mesh_m$.
Then there exists $\bar\rho$ with $\rho_{\min}\leq\bar \rho\leq\rho_{\max}$
and $\bar \bfu \in L^{2}(0,T; \boldsymbol{E}(\Omega))$  such that, up to a subsequence:
\begin{list}{-}{\itemsep=0.ex \topsep=0.5ex \leftmargin=1.cm \labelwidth=0.7cm \labelsep=0.3cm \itemindent=0.cm}
\item the sequence $(\bfu_m)_\mnn$ converges to $\bar \bfu$ in $L^{2}(0,T; L^{2}(\Omega)^{d})$,
\item the sequence $(\rho_m)_\mnn$ converges to $\bar \rho$ in $ \in L^{2}(0,T;L^2(\Omega))$,
\item $(\bar \rho,\bar \bfu)$ is a solution to the weak formulation \eqref{weak-mass} and \eqref{weak-mom}.
\end{list}
\end{theorem}

\noindent \textbf{Sketch of proof:}
\begin{list}{--}{\itemsep=0.5ex \topsep=0.5ex \leftmargin=1.cm \labelwidth=0.3cm \labelsep=0.5cm \itemindent=0.cm}
\item Thanks to \eqref{bound-density}, there exists a subsequence of $(\rho_m)_{m\in \xN}$  star-weakly converging to some $\bar\rho$ in $L^\infty(\Omega\times(0,T))$; thanks to \eqref{estiLdeux:ro} and \eqref{translates}, there exists a subsequence of $(u_m)_{m\in \xN}$  converging to some $\bar u$ in $L^2(0,T;(L^2(\Omega)^d)$.
\item  Passing to the limit in   \eqref{scheme-mass} yields that $(\bar \rho, \bar u) $ satisfies \eqref{weak-mass}.
\item The strong  convergence of the approximate densities is then obtained thanks to the $L^2$ estimates for $\rho$ in both the discrete and continuous case\ \cite[Proposition 8.7]{LS_densitevar}.
\item Passing to the limit in  \eqref{scheme-mom} yields that  $(\bar \rho, \bar u) $ satisfies \eqref{weak-mom}.
\item We finally obtain that $\bar \bfu\in L^2(0,T;\bfE(\Omega))$, where $E(\Omega) = \{\bfv \in H^1_0(\Omega)$ s.t. $\dive \bfv = 0\}$, as in \cite[Proof of Theorem 4.3]{ghlm-16-macns}.
\end{list}
%
% --------------------------------------------------------------------------------------------------------------------
% 
\bibliographystyle{plain}
\bibliography{./mac}
\end{document}